\documentclass[a4paper,11pt]{amsart}
\usepackage{amsmath, amssymb, amsthm, mathtools}
\usepackage{hyperref}
\usepackage{graphicx}

\usepackage{verbatim}
\usepackage{enumitem}
\allowdisplaybreaks

\newcounter{theorem}
\newtheorem{thm}[theorem]{Theorem}
\newtheorem*{claim}{Claim}
\newtheorem{lemma}[theorem]{Lemma}
\newtheorem{prop}[theorem]{Proposition}
\newtheorem{cor}[theorem]{Corollary}

\usepackage{perpage}
\newcounter{mparcnt}
\MakePerPage{mparcnt}

\theoremstyle{definition}
\newtheorem{defn}[theorem]{Definition}

\theoremstyle{remark}
\newtheorem*{remark*}{Remark}
\newtheorem{remark}[theorem]{Remark}

\theoremstyle{plain}
\newcounter{theoremintro}
\newtheorem{thmIntro}[theoremintro]{Theorem}
\newtheorem{corIntro}[theoremintro]{Corollary}

\numberwithin{equation}{section}
\numberwithin{theorem}{section}

\newcommand{\e}{\epsilon}
\newcommand{\dl}{\delta}

\newcommand{\suppo}{\mathrm{supp}^{\circ}}

\newcommand{\N}{\mathbb{N}}

\newcommand{\supp}{\mathrm{supp}}

\renewcommand{\setminus}{\backslash}

\newcommand{\cU}{\mathcal{U}}

\newcommand{\TODO}[1]{}

\begin{document}

\title[uniform property \texorpdfstring{$\Gamma$}{Gamma} and the small boundary property]{Uniform Property $\Gamma$ and the Small Boundary Property}

\thanks{G.\ K.\ and A.\ V.\  were supported by the Deutsche Forschungsgemeinschaft (DFG, German Research Foundation) under Germany’s Excellence Strategy EXC 2044--390685587, Mathematics Münster: Dynamics--Geometry--Structure, through SFB 1442 and ERC Advanced Grant 834267--AMAREC. A.\ T.\ was supported by an NSERC Discovery Grant and by a Visiting Research Fellowship at All Souls College.}

\author[G. Kopsacheilis]{Grigoris Kopsacheilis}
\address{Grigoris Kopsacheilis, Mathematisches Institut, Fachbereich Mathematik und Informatik der
Universit\"at M\"unster, Einsteinstrasse 62, 48149 M\"unster, Germany.}
\email{gkopsach@uni-muenster.de}
\urladdr{https://sites.google.com/view/gkopsach}

\author[H. Liao]{Hung-Chang Liao}
\address{Hung-Chang Liao, Department of Mathematics and Statistics, University of Ottawa, 150 Louis-Pasteur Pvt, Ottawa, ON, Canada K1N 6N5}
\email{hliao@uottawa.ca}

\author[A. Tikuisis]{Aaron Tikuisis}
\address{Aaron Tikuisis, Department of Mathematics and Statistics, University of Ottawa, 150 Louis-Pasteur Pvt, Ottawa, ON, Canada K1N 6N5}
\email{aaron.tikuisis@uottawa.ca}

\author[A. Vaccaro]{Andrea Vaccaro}
\address{Andrea Vaccaro, Mathematisches Institut, Fachbereich Mathematik und Informatik der
Universit\"at M\"unster, Einsteinstrasse 62, 48149 M\"unster, Germany.}
\email{avaccaro@uni-muenster.de}
\urladdr{https://sites.google.com/view/avaccaro}


\begin{abstract}
We prove that, for a free action $\alpha \colon G \curvearrowright X$ of a countably infinite discrete amenable group on a compact metric space, the small boundary property is implied by uniform property $\Gamma$ of the Cartan subalgebra $(C(X) \subseteq C(X) \rtimes_\alpha G)$. The reverse implication has been demonstrated by Kerr and Szab\'o,
from which we obtain that these two conditions are equivalent. We moreover show that, if $\alpha$ is also minimal, then almost finiteness of $\alpha$ is implied by
tracial $\mathcal{Z}$-stability of the subalgebra $(C(X) \subseteq C(X) \rtimes_\alpha G)$. 
The reverse implication is due to Kerr, resulting in the equivalence of these two properties as well. As an application, we prove that if $\alpha \colon G \curvearrowright X$ and $\beta \colon H \curvearrowright Y$ are free actions and $\alpha$ has the small boundary property, then $\alpha \times \beta \colon G \times H \curvearrowright X \times Y$ has the small boundary property. An analogous permanence property is obtained for almost finiteness in case $\alpha$ and $\beta$ are free minimal actions.
\end{abstract}

\maketitle

\section{Introduction}
The adaptation of the theory surrounding the Toms--Winter conjecture to the framework of topological dynamics, developed by Kerr in \cite{Kerr:JEMS}
and Kerr--Szab\'o in \cite{KerrSzabo}, has revealed a series of correlations between $C^\ast$-algebraic properties that have arose in connection to the Elliott classification program, and specific aspects of regularity within dynamical systems, such as almost finiteness and the small boundary property.
The goal of this paper is to establish a formal correspondence between some of these properties.

The Toms--Winter conjecture is a $C^\ast$-algebraic statement whose primary purpose is to isolate a robust notion of regularity which, alongside nuclearity (and modulo a thorny issue called the Universal Coefficient Theorem), abstractly identifies the family of \emph{classifiable} $C^\ast$-algebras (see \cite{WinterICM, White:survey} for an overview on the topic).
The conjecture compares three properties: \emph{finite nuclear dimension}, \emph{$\mathcal{Z}$-sta\-bi\-li\-ty}, and \emph{strict comparison of positive elements}, and claims that, for simple separable nuclear non-elementary $C^\ast$-algebras, these are equivalent, all being aspects of the same overarching form of regularity
(see \cite[Section 5]{WinterICM} and \cite[Section 5]{CETW:IMRN} for a complete discussion on this conjecture).

After the breakthrough towards the Toms--Winter conjecture made in \cite{CETWW}, a notion that was recognized to be important in this setting is \emph{uniform property $\Gamma$}, inspired by Murray and von Neumann's property $\Gamma$ for II$_1$ factors.
We have in fact the equivalence of the following two conditions, for simple unital separable non-elementary nuclear $C^\ast$-algebras (see \cite{CETW:IMRN}):
\begin{enumerate}[label=(\arabic*)]
\item \label{item:TW1} $\mathcal{Z}$-stability,
\item \label{item:TW2} strict comparison and uniform property $\Gamma$.
\end{enumerate}
Both (uniform) property $\Gamma$ and $\mathcal Z$-stability are central sequence properties: that is, they each express a certain richness in the central sequences.

In \cite{Kerr:JEMS}, Kerr proposed correspondents to the properties considered in the Toms--Winter conjecture within the framework of actions of amenable groups on compact metric spaces.
Of these, \emph{almost finiteness} has been particularly influential.
It is introduced as a topological adaptation of Ornstein--Weiss' formulation of hyperfiniteness for free probability measure-preserving actions in terms of tilings of the space, with a ``tracial smallness'' aspect (see \cite{Kerr:JEMS} for a thorough presentation).
The tracial smallness condition nods to Lin's tracial approximation theory (see \cite{Lin:TAMS,Lin:DMJ,GLN:survey} for example), and more pertinently, ties to a characterization of $\mathcal Z$-stability for \emph{nuclear} $C^*$-algebras enabled by Matui and Sato (called \emph{tracial $\mathcal Z$-stability} by Hirshberg and Orovitz, \cite{HirshbergOrovitz}).
Kerr proved that almost finite free minimal actions of countably infinite amenable groups give rise to $\mathcal{Z}$-stable crossed products, and moreover showed that almost finiteness implies \emph{dynamical comparison}, a notion of comparison for dynamical systems originally introduced by Winter.
In \cite{Kerr:JEMS}, Kerr presents almost finiteness as an analogue of tracial $\mathcal Z$-stability; it is noteworthy that the use of group amenability in arguments (via tiling methods) somewhat muddle where the central sequences truly come from.

Nonetheless, the one-way implication from almost finiteness to $\mathcal Z$-stability, together with conceptual similarities between the notions, naturally raises the question of a converse -- a question that remains open.
Some light on this, however, began to be shed by Kerr and Szab\'o in \cite{KerrSzabo}, containing a detailed analysis of the \emph{small boundary property}: a dynamical formulation of zero-dimensionality introduced  by Shub and Weiss in \cite{ShubWeiss:ETDS} and investigated more in detail in Lindenstrauss and Weiss's study of Gromov's mean dimension \cite{LindenstraussWeiss}.
In \cite[Theorem~6.1]{KerrSzabo} it is demonstrated that the relationship between almost finiteness, dynamical comparison, and the small boundary property mirrors the equivalence between \ref{item:TW1} and \ref{item:TW2} mentioned above; more precisely, the following are equivalent for free actions of amenable groups:
\begin{enumerate}[label=(\Roman*)]
\item almost finiteness,
\item dynamical comparison and the small boundary property.
\end{enumerate}
Moreover, they proved (again using Ornstein--Weiss tiling) that crossed products originating from free actions with the small boundary property have uniform property $\Gamma$, promoting the apparent analogy between these two notions to a formal, yet one-way, relationship.

The main result of this paper shows that this inference is just one facet of a reciprocal relationship. 
It was noted in \cite{KerrSzabo} that the small boundary property implies a stronger condition -- uniform property $\Gamma$ for the Cartan sub-$C^\ast$-algebra $(C(X) \subseteq C(X) \rtimes_\alpha G)$.
Property $\Gamma$ for Cartan subalgebras is not a new concept; indeed, this property was famously used by Connes and Jones to distinguish two Cartan subalgebras, providing the first example of a II$_1$ factor with non-isomorphic Cartan subalgebras (\cite{ConnesJones}).
This was a precursor to much of the deformation/rigidity theory such as \cite{OzawaPopa}, where the opposite phenomenon (unique Cartan subalgebras, even up to conjugation) is proved in a class of II$_1$ factors.

Our main theorem, as follows, establishes a converse to the aforementioned result in \cite{KerrSzabo}, confirming a speculation made there by Kerr and Szab\'o.

\begin{thmIntro}[Theorem \ref{thm:Main2}]
\label{thm:GammaSBP}
Let $X$ be a compact metric space, let $G$ be a countably infinite discrete amenable group and let $\alpha \colon G\curvearrowright X$ be a free action.
The following are equivalent:
\begin{enumerate}
\item \label{item1:main} $\alpha$ has the small boundary property.
\item \label{item2:main} $(C(X)\subseteq C(X)\rtimes_\alpha G)$ has uniform property $\Gamma$.
\end{enumerate}
\end{thmIntro}

One important consequence of Theorem \ref{thm:GammaSBP} is that it can be applied, in combination with the main result of \cite{LiaoTikuisis}, to establish a rigorous connection
between almost finiteness and $\mathcal{Z}$-stability. We show in fact that the $C^\ast$-algebraic equivalent of almost finiteness is \emph{tracial $\mathcal{Z}$-stability} for sub-$C^\ast$-algebras (Definition \ref{def:trZstab}), introduced in \cite{LiaoTikuisis} and based on Hirshberg--Orovitz's tracial $\mathcal Z$-stability for $C^*$-algebras.

\begin{thmIntro}[Corollary \ref{cor:AlmostFinite}]\label{thm:AlmostFinite}
Let $X$ be a compact metric space, let $G$ be a countably infinite discrete amenable group and let $\alpha \colon G \curvearrowright X$ be a free minimal action.
The following are equivalent:
\begin{enumerate}
\item\label{it:AlmostFinite.1intro}
 The action $\alpha$ is almost finite.
\item\label{it:AlmostFinite.2intro}
$(C(X)\subseteq C(X)\rtimes_\alpha G)$ is tracially $\mathcal Z$-stable.
\item\label{it:AlmostFinite.3intro}
$(C(X)\subseteq C(X)\rtimes_\alpha G)$ has uniform property $\Gamma$ and $\alpha$ has dynamical comparison.
\item\label{it:AlmostFinite.4intro}
$\alpha$ has the small boundary property and dynamical comparison.
\end{enumerate}
\end{thmIntro}

The main proof of this paper is that of \ref{item2:main} $\Rightarrow$ \ref{item1:main} in Theorem \ref{thm:GammaSBP}.
It consists of several steps and, as a byproduct that is interesting in its own right, it factors through an adaptation of the  $C^\ast$-algebraic concept of \emph{complemented partitions of unity} (typically abbreviated as \emph{CPoU}) to sub-$C^\ast$-algebras (see Definition \ref{def:CPoU}).
CPoU first appeared in \cite{CETWW}, and they have served since as a device for certain local-to-global arguments (see \cite[Section~1.4]{CCEGSTW} for an exposition).

The first part of our proof holds in greater generality, and it applies to sub-$C^\ast$-algebras $(D \subseteq A)$ which are either Cartan pairs, or are of the form $(C(X) \subseteq C(X)\rtimes_\alpha G)$, where $\alpha \colon G \curvearrowright X$ is an action of a countable discrete amenable group on a compact metric space $X$. Note that both these settings cover the case of $(C(X)\subseteq C(X)\rtimes_\alpha G)$ where $\alpha \colon G \curvearrowright X$ is a free action of an amenable group.
Under either of these hypotheses, we show that if $(D \subseteq A)$ has uniform property $\Gamma$, then it has CPoU (see Corollary \ref{cor:GammaCPoU}), in analogy to what happens in the classical setting (\cite[Theorem~3.8]{CETWW}).
The proof splits in two steps: we first establish \emph{weak CPoU}, a weakening of CPoU where the partitions of unity are allowed to be positive, non-orthogonal elements (as opposed to orthogonal projections; see Definition \ref{def:weakCPoU}). This is done with two independent proofs for the two cases, the first one relying on averaging argument using F{\o}lner sets (Theorem~\ref{thm:MainWeakCPoU}), while the second uses the canonical conditional expectation $E \colon A \to D$ to push partitions of unity from $A$ onto $D$ (Theorem~\ref{thm:weak-cpou-cartan-pairs}).
After this, a routine \emph{projectionization} trick, enabled by uniform  property $\Gamma$, permits to upgrade this weaker version of CPoU to actual CPoU (Proposition \ref{prop:weakCPoUplusGamma}).

Once this is done, the aforementioned local-to-global argument powered by CPoU can be used to verify a formulation of the small boundary property due to Kerr and Szab\'o, in terms of disjoint open almost-covers (covering the space except a piece that is uniformly small with respect to all invariant measures) by small open sets.

During the preparation of this paper we learned that Elliott and Niu independently proved the implication \ref{item2:main} $\Rightarrow$ \ref{item1:main} in Theorem \ref{thm:GammaSBP} (\cite{ElliottNiu}).
Whereas they use a more direct topological approach, the connection to CPoU in our argument may be of independent interest, for example in further investigations into Cartan subalgebras.

We conclude with some applications of our main results. The first one exploits the stability of uniform property $\Gamma$ and of tracial $\mathcal{Z}$-stability with respect to the tensor product operation. Thanks to the equivalences in Theorems \ref{thm:GammaSBP} and \ref{thm:AlmostFinite}, we can derive similar patterns for the small boundary property and almost finiteness.

\begin{corIntro}[Corollary~\ref{cor:Gamma_product}, Corollary~\ref{cor:AlmostFinite_product}]
\label{cor:SBPproducts}
Let $X$ and $Y$ be compact metric spaces, let $G$ and $H$ be countably infinite discrete amenable groups and let $\alpha\colon G \curvearrowright X$ and $\beta\colon G\curvearrowright Y$ be free actions.
\begin{enumerate}
\item \label{item1:corollary} If the action $\alpha$ has the small boundary property, then so does the action $\alpha\times\beta \colon(G\times H)\curvearrowright (X\times Y)$.
\item If the actions $\alpha$ and $\beta$ are moreover minimal and $\alpha$ is almost finite, then $\alpha\times\beta \colon(G\times H)\curvearrowright (X\times Y)$ is almost finite.
\end{enumerate}
\end{corIntro}
While this paper was under review, results similar to that in \ref{item1:corollary} of the corollary above were obtained by Kerr and Li in the context of actions of not necessarily amenable groups (\cite{KerrLi}).

Finally, we note that our main result, combined with appropriate dynamical examples, yields sub-$C^\ast$-algebras without uniform property $\Gamma$; in their paper introducing the small boundary property, Lindenstrauss and Weiss showed that it implies mean dimension zero (\cite[Theorem~5.4]{LindenstraussWeiss}; the proof is for $\mathbb Z$-actions but generalizes to arbitrary amenable groups).

\begin{corIntro} \label{cor:noGamma}
Let $\alpha \colon G \curvearrowright X$ be a free action of an amenable group with positive mean dimension.
Then the sub-$C^\ast$-algebra $(C(X)\subseteq C(X)\rtimes_\alpha G)$ does not have uniform property $\Gamma$.
\end{corIntro}

In particular, the free minimal actions $\alpha \colon \mathbb Z \curvearrowright X$ constructed in \cite[Section~2]{GiolKerr} (which give examples of non-$\mathcal Z$-stable $C^*$-algebras arising as $\mathbb Z$-crossed products) are tailored to have positive mean dimension, so we find that $(C(X)\subseteq C(X)\rtimes_\alpha \mathbb Z)$ does not uniform property $\Gamma$.
Corollary \ref{cor:noGamma} does not answer the question of whether the $C^*$-algebra $C(X)\rtimes_\alpha \mathbb Z$ itself has uniform property $\Gamma$.

\subsection*{Summary of the paper} Section \ref{sec:Prelim} is devoted to preliminaries. Section \ref{sec:GammaSBP} focuses on the small boundary property and uniform property $\Gamma$ for sub-$C^\ast$-algebras. In Section \ref{sec:CPoU} we introduce CPoU for sub-$C^\ast$-algebras, and we show that if $(C(X)\subseteq C(X)\rtimes_\alpha G)$ is a sub-$C^\ast$-algebra with CPoU and originating from a free action $\alpha \colon G \curvearrowright X$, then $\alpha$ has the small boundary property.
Finally in Section \ref{sec:Main} we put all ingredients together and prove Theorems \ref{thm:GammaSBP} and \ref{thm:AlmostFinite}, as well as Corollary \ref{cor:SBPproducts}.

\subsection*{Acknowledgements} We are grateful to David Kerr, Stuart White and Wilhelm Winter for helpful discussions and comments. We thank moreover the anonymous referee for their careful feedback revealing some imprecisions in an earlier version of the paper.
\section{Preliminaries}
\label{sec:Prelim}

This brief preliminary section is devoted to notation and definitions.

For a $C^\ast$-algebra $A$, we write $A_+$ for the set of positive elements in $A$, and we write $A_+^1$ for the set of positive contractions in $A$. The \emph{trace space} $T(A)$ is the set of all
tracial states on $A$, which we simply refer to as \emph{traces}. For $a,b \in A$, we write $[a,b]$ for the commutator $ab - ba$.

Given a nonempty subset $X\subseteq T(A)$, we define the seminorm $\|\cdot\|_{2,X}$ on $A$ by
\begin{equation}   \|a\|_{2,X} := \sup_{\tau\in X} \tau(a^*a)^{ \frac{1}{2} }.  \end{equation} 
Throughout the paper, $\mathcal U$ denotes a (fixed) free ultrafilter on $\N$. 
The \emph{tracial ultrapower of $A$ with respect to $X$} is the quotient
\begin{equation} 
(A,X)^\cU := \ell^\infty(A) /\{ (a_n)_{n=1}^\infty \in \ell^\infty(A): \lim_{n\to \cU} \| a_n\|_{2,X} = 0  \}.
\end{equation} 
When $X = T(A)$ we recover the usual \emph{uniform tracial power} $(A,T(A))^\cU$, abbreviated as $A^\cU$. We recall that, in the case that $A$ is separable, the tracial ultrapower $A^\cU$ is unital if and only if $T(A)$ is compact. In this case, a representative sequence for $1_{A^\cU}$ can be obtained by taking any countable approximate unit of $A$ (see  \cite[Proposition~1.11]{CETWW}).

A tracial state on $A^\cU$ that is induced by a sequence $(\tau_n)_{n=1}^\infty$ of tracial states on $A$ is called a \emph{limit trace}. Following \cite[\S 1.4]{CETWW} we write $T_\cU(A)$ for the set of limit traces on $A^\cU$.

	A \emph{sub-$C^\ast$-algebra} $(D\subseteq A)$ is a $C^\ast$-algebra $A$ together with a nonzero $C^\ast$-subalgebra $D$. We say that a sub-$C^\ast$-algebra $(D\subseteq A)$ is \emph{nondegenerate} if $D$ contains an approximate unit for $A$, and it is \emph{unital} if $A$ is unital and $1_A \in D$. Every tracial state on $A$ restricts to a positive tracial functional on $D$. We write $T(A)|_D \subseteq T(D)$ for the set of tracial states on $D$ obtained from restricting and scaling elements of $T(A)$ (if the sub-$C^\ast$-algebra $(D\subseteq A)$ is nondegenerate then no scaling is needed). 
We let 
\begin{equation} \label{eq:kappa}
\kappa \colon 
D^\cU \to A^\cU
\end{equation}
denote the $^\ast$-homomorphism sending the class of $(d_n)_{n=1}^\infty \in \ell^\infty(D)$ to the class of the same sequence in $A^\cU$.\footnote{This map is well-defined since
$\| d \|_{2, T(A)} \le \| d \|_{2, T(D)}$ for all $d \in D$, and therefore a similar inequality is true for limit traces.}
The image $\kappa(D^\cU)\subseteq A^\cU$ consists of all elements in $A^\cU$ that are representable by sequences in $\ell^\infty(D)$. Note that in case $T(A)$ is compact and $(D \subseteq A)$ is nondegenerate, the unit $1_{A^\cU}$ of $A^\cU$ lies in the image of $\kappa$.

Given two subsets $B,S \subseteq (A,X)^\cU$ we let $B \cap S'$ denote the relative commutant of $S$ in $B$.
While in general the map sending each $a \in A$ to the corresponding constant sequence in $(A,X)^\cU$ is not an embedding (this is the case only if $\| \cdot \|_{2, X}$ is a norm), we still write $a$ for the corresponding constant sequence in $(A, X)^\cU$, and denote by $B \cap A'$ the relative commutant of the set of all constant sequences.

For an action $\alpha \colon G \curvearrowright A$ on a $C^\ast$-algebra $A$ with nonempty trace space, we let $\alpha^\cU$ denote the action induced on $A^\cU$, which acts coordinate-wise as $\alpha$ on each representing sequence.

For a continuous action $G \curvearrowright X$ on a topological space $X$, we let $M_G(X)$ denote the set of $G$-invariant Borel probability measures on $X$.

\section{The small boundary property and uniform property \texorpdfstring{$\Gamma$}{Gamma}} \label{sec:GammaSBP}
In this section we recall the definition of the small boundary property and introduce uniform property $\Gamma$ for sub-$C^\ast$-algebras. We also collect some preliminary statements about these two notions.

\subsection{The small boundary property}

The small boundary property is generally considered in the context of group actions on compact metrizable spaces. Since some of our results concern more general sub-$C^\ast$-algebras than those arising from crossed products, we formulate a definition of the small boundary property that does not rely on the existence of an action, but that only depends on a given set of measures of the space. A similar notion is also considered in
\cite{ElliottNiu}.
\begin{defn}\label{def:sbp-general}
Let $X$ be a compact Hausdorff space and let $K$ be a weak$^\ast$ compact set of Borel probability measures on $X$. We say that $X$ has the $K$-\textit{small boundary property} if the collection
\begin{equation}  
\left\{U\subseteq X\colon U\text{ is open and }\mu(\partial U)=0, \, \forall \mu \in K\right\} 
\end{equation} 
is a basis for the topology on $X$.

\end{defn}

For a continuous action $\alpha \colon G \curvearrowright X$ of a discrete group on a compact Hausdorff space, the action has the \textit{small boundary property} (as in \cite[Definition~5.1]{KerrSzabo}) if $X$ has the $M_G(X)$-small boundary property. We give an analogous definition for sub-$C^\ast$-algebras.

\begin{defn}
Let $(D\subseteq A)$ be a unital sub-$C^\ast$-algebra with $D$ abelian and such that $A$ has nonempty trace space. Denoting by $X$ the spectrum of $D$, we say that $(D\subseteq A)$ has the \textit{small boundary property} if $X$ has the $T(A)\vert_D$-small boundary property, where $T(A)\vert_D$ is identified with the corresponding compact set of Borel probability measures on $X$.
\end{defn}

\begin{remark}\label{rmk:sbp of crossed product implies usual sbp}
For an action $\alpha\colon G \curvearrowright X$ of an amenable group, the small boundary property of $(C(X)\subseteq C(X)\rtimes_\alpha G)$ is equivalent to $\alpha$ having the small boundary property, since any invariant measure $\mu\in M_G(X)$ is obtained as the restriction on $C(X)$ of the trace $\tau_\mu\in T(C(X)\rtimes_\alpha G)$ defined as $\tau_\mu(\cdot)\coloneqq \int_XE(\cdot)d\mu$, where $E\colon C(X)\rtimes_\alpha G \to C(X)$ denotes the canonical conditional expectation.

\end{remark}

We rely on the following characterization of the small boundary property, due to Kerr and Szab{\'o}.
(In \cite{KerrSzabo}, they work in the setting of group actions, but their argument generalizes straightforwardly, as we explain below.)

\begin{prop}[{cf.\ \cite[Theorem~5.5]{KerrSzabo}}] \label{prop:sbp}
Let $X$ be a compact metric space and $K$ a compact set of Borel probability measures on $X$. Then, the following are equivalent:
\begin{enumerate}
\item\label{sbpitem1} $X$ has the $K$-small boundary property.
\item\label{sbpitem2} For any $\delta>0$ there is a finite collection $\mathcal{I}$ of pairwise disjoint open subsets of $X$ such that each $U\in\mathcal{I}$ has diameter at most $\delta$ and $\mu(\bigsqcup_{U\in\mathcal{I}}U)=1$ for all $\mu\in K$.
\item\label{sbpitem3} For any $\delta,\epsilon>0$ there is a finite collection $\mathcal{I}$ of pairwise disjoint open subsets of $X$ such that each $U\in\mathcal{I}$ has diameter at most $\delta$ and $\mu(\bigsqcup_{U\in\mathcal{I}}U)>1-\epsilon$ for all $\mu\in K$.
\end{enumerate}
\end{prop}

\begin{proof}
For \ref{sbpitem1}$\Rightarrow$\ref{sbpitem2}, let $\delta>0$. By compactness and the $K$-small boundary property, we obtain a finite open cover of $X$, say $X=\bigcup_{j\in J}V_j$, where each open set $V_j$ has diameter at most $\delta$ and satisfies $\mu(\partial V_j)=0$ for all $\mu\in K$. For $S\subseteq J$ set
\begin{equation}
U_S \coloneqq \bigcap_{j\in S}V_j\cap \bigcap_{j\not\in S}(X\setminus\overline{V}_j).
\end{equation}
Then the collection $\mathcal{I} \coloneqq \{U_S: S\subseteq J\}$ is directly seen to satisfy the desired properties.

The implication \ref{sbpitem2}$\Rightarrow$\ref{sbpitem3} is trivial, and \ref{sbpitem3}$\Rightarrow$\ref{sbpitem1} follows from the same argument as in the proof of \cite[Theorem~5.5]{KerrSzabo} (iv)$\Rightarrow$(v)$\Rightarrow$(i), but one needs to replace the use of Proposition~3.4 therein with a standard application of Dini's theorem. 
\end{proof}

\subsection{Uniform property \texorpdfstring{$\Gamma$}{Gamma}}
\begin{defn} \label{def:Gamma}
Let $(D\subseteq A)$ be a sub-$C^\ast$-algebra with $A$ separable and $T(A)$ nonempty and compact.
We say that $(D\subseteq A)$ has \emph{uniform property $\Gamma$} if the sub-$C^\ast$-algebra $(\kappa(D^\cU) \subseteq A^\cU)$ is unital (with $\kappa$ as in \eqref{eq:kappa}) and if for any $k \in \N$ there exists a partition of unity of $A^\cU$ consisting of projections $q_1, \dots, q_k \in \kappa(D^\cU) \cap A'$, meaning that $\sum_{i=1}^k q_i = 1_{A^\cU}$, such that 
\begin{equation} \label{eq:Gamma} 
\tau(q_ia) = \frac1k \tau(a), \quad \tau \in T_\cU(A), \, a \in D, \, i = 1, \dots, k. 
\end{equation}
\end{defn}

\begin{remark} \label{remark:unital}
When $T(A)$ is compact, the sub-$C^\ast$-algebra $(\kappa(D^\cU) \subseteq A^\cU)$ is automatically unital if, for instance, $(D \subseteq A)$ is nondegenerate (\cite[Proposition~1.11]{CETWW}). An even weaker condition still ensuring this is the existence in $D$ of what could be called a \emph{tracial approximate unit}, i.e., a sequence of positive elements $(d_n)_{n=1}^\infty$ such that $\lim_{n \to \infty} \| d_n a - a\|_{2, T(A)} = \lim_{n \to \infty} \| a d_n - a \|_{2, T(A)} = 0$, for all $a \in A$.
\end{remark}

\begin{remark} \label{remark:expectation}
Definition~\ref{def:Gamma} generalizes \cite[Definition~2.1]{CETWW}: a $C^\ast$-algebra $A$ has uniform property $\Gamma$ as defined there if and only if the sub-$C^\ast$-algebra $(A\subseteq A)$ has uniform property $\Gamma$.

On the other hand, if $(D \subseteq A)$ has uniform property $\Gamma$, it is not true a priori that $A$ has uniform property $\Gamma$, since condition \eqref{eq:Gamma} is only required to hold for elements in $D$.
This choice is motivated by our definition of complemented partitions of unity for sub-$C^\ast$-algebras, introduced in the next section (see Definition~\ref{def:CPoU}).
Nevertheless, if $(D\subseteq A)$ has uniform property $\Gamma$ and if there is a conditional expectation $E \colon A \to D$ such that every trace $\tau \in T(A)$ factors through $E$ (that is $\tau = \tau \circ E$),
then \eqref{eq:Gamma} holds automatically for all $a \in A$, and therefore $A$ has uniform property $\Gamma$. This is what happens for instance in our main case of interest: $C^\ast$-diagonals (see \cite{Kumjian:CJM,Renault:IMSB} and \cite[Lemma~4.3]{LiRenault}).
\end{remark}

\begin{remark}
If $(D\subseteq D\rtimes_\alpha G)$ is a sub-$C^\ast$-algebra originating from an action $\alpha$ of a discrete amenable group $G$ on a unital $C^\ast$-algebra $D$, then Definition~\ref{def:Gamma} can be formulated without reference to the crossed product:\footnote{For an action $G \curvearrowright S$ of a group $G$ on a set $S$, we use $S^G$ to denote the subset of $S$ consisting of elements fixed by every $g\in G$.} the definition asks for a partition of projections $q_1,\dots,q_k \in ((D,T(D)^G)^\cU)^G \cap D'$, such that \eqref{eq:Gamma} holds for all limit traces coming from $T(D)^G$.

Note that in   \cite[Definition~3.1]{GHV} a notion of \emph{equivariant} uniform property $\Gamma$ for actions on unital $C^\ast$-algebras is defined, which is the same as the formulation just given, with the key differences of requiring $q_1,\dots,q_k \in (D^\cU)^G \cap D'$ and for \eqref{eq:Gamma} to hold for all limit traces originating from elements of $T(D)$, instead of the subset $T(D)^G$ (a generalization of equivariant uniform property $\Gamma$ to the non-unital case can be found in \cite[Definition~2.1]{SzaboWouters}, and a variant of it only depending on invariant traces already appeared in \cite[Definition~4.1]{ESPA}).
This {equivariant} formulation of property $\Gamma$ is much stronger than Definition~\ref{def:Gamma}, and it implies that $D$ itself has uniform property $\Gamma$.
In particular, it never holds when $D$ is commutative.
\end{remark}

Definition~\ref{def:Gamma} can be upgraded, by a routine reindexing argument, to a version where the projections $q_1,\dots,q_k$ are required to commute with an arbitrary given $\| \cdot \|_{2, T_\cU(A)}$-separable subset of $A^\cU$, and where \eqref{eq:Gamma} holds for all elements in an arbitrary $\| \cdot \|_{2,T_\cU(A)}$-separable subset of $\kappa(D^\cU)$. We summarize this fact in the following lemma, whose proof is omitted (see e.g.\ \cite[Remark~5.15]{CCEGSTW} or \cite[Proposition~4.4]{Winter12}).

\begin{lemma} \label{lemma:Gamma+}
Let $(D\subseteq A)$ be a sub-$C^\ast$-algebra with $A$ separable and $T(A)$ nonempty and compact, and suppose that $(D \subseteq A)$ has uniform property $\Gamma$. Then
for every $k \in \N$, every $\| \cdot \|_{2, T_\cU(A)}$-separable subset $S$ of $A^\cU$ and every $\| \cdot \|_{2, T_\cU(A)}$-separable subset $T$ of $\kappa(D^\cU)$, there exists a partition of unity of $A^\cU$ consisting of projections $q_1, \dots, q_k \in \kappa(D^\cU) \cap S'$ such that
\begin{equation} 
 \tau(q_ia) = \frac1k \tau(a), \quad \tau \in T_\cU(A), \, a \in T, \, i = 1, \dots, k. 
\end{equation} 
\end{lemma}

We will need the following technical consequence of uniform property $\Gamma$ for a sub-$C^\ast$-algebra, which is a generalization of the \emph{tracial projectionization} result obtained in \cite[Lemma~2.4]{CETWW}.

\begin{lemma}\label{lem:projectionization}
Let $(D\subseteq A)$ be a sub-$C^\ast$-algebra with $A$ separable and $T(A)$ nonempty and compact, and let $S \subseteq A^\cU$ and $T\subseteq \kappa(D^\cU)$ be $\|\cdot\|_{2,T_\cU(A)}$-separable subsets. 
If $(D\subseteq A)$ has uniform property $\Gamma$, then for any positive contraction $b \in \kappa(D^\cU) \cap S'$, there is a projection $p \in \kappa(D^\cU) \cap S'$ such that
\begin{equation}  \tau(pa)=\tau(ba), \quad a \in T,\, \tau\in T_\cU(A). \end{equation} 
\end{lemma}
 
\begin{proof} 
This is proved following verbatim the proof of \cite[Lemma~2.4]{CETWW}, with Lemma~\ref{lemma:Gamma+} replacing the use of \cite[Lemma~2.2]{CETWW}. The element $p$ defined right before \cite[Equation~(2.9)]{CETWW} belongs to $\kappa(D^\cU)\cap S'$ because $b$ and the projections $q_1,\dots,q_k$ obtained from uniform property $\Gamma$ do. 
\end{proof}

\section{Complemented partitions of unity} \label{sec:CPoU}
We begin this section by introducing a version of complemented partitions of unity for sub-$C^\ast$-algebras, which we then use to deduce the small boundary property in Proposition~\ref{prop:CPoU_SBP}. 

\begin{defn} \label{def:CPoU}
Let $(D\subseteq A)$ be a sub-$C^\ast$-algebra, with $A$ separable and $T(A)$ nonempty and compact.
We say that $(D\subseteq A)$ has \emph{complemented partitions of unity} (abbreviated as \emph{CPoU}) if the sub-$C^\ast$-algebra $(\kappa(D^\cU) \subseteq A^\cU)$ is unital (with $\kappa$ as in \eqref{eq:kappa}) and if for any $\|\cdot\|_{2,T_\cU(A)}$-separable subset $S$ of $A^\cU$, for any $a_1,\dots,a_k \in D_+$ and $\delta > 0$ satisfying
\begin{equation}\label{eq:CPoUinput} \sup_{\tau \in T(A)} \min\{\tau(a_1),\dots,\tau(a_k)\} < \delta, \end{equation}
there exists a partition of unity of $A^\cU$ consisting of projections $p_1,\dots,p_k \in \kappa(D^\cU) \cap S' $ such that
\begin{equation}
\tau(a_ip_i)\leq \dl\tau(p_i), \quad \tau \in T_\cU(A), \, i=1,\dots,k.
\end{equation} 
\end{defn}

Observations analogous to those in Remarks~\ref{remark:unital} and \ref{remark:expectation} hold for Definition~\ref{def:CPoU}. In particular, if there is a conditional expectation $E \colon A \to D$ such that
$\tau = \tau \circ E$ for every $\tau \in T(A)$ (e.g.\ when $(D \subseteq A)$ is a $C^\ast$-diagonal), then Definition~\ref{def:CPoU} is equivalent to its formal strengthening where $a_1,\dots,a_k$ are allowed to range in $A_+$ (instead of only $D_+$). In this case we get that if $(D \subseteq A)$ has CPoU, then $(A \subseteq A)$ has CPoU as well, meaning that $A$ has CPoU in the sense of \cite[Definition~3.1]{CETWW}. 

A similar remark holds for the following definition, which is a weakening of CPoU and a technical intermediate notion needed for our main result.
\begin{defn}\label{def:weakCPoU}
Let $(D\subseteq A)$ be a sub-$C^\ast$-algebra, with $A$ separable and $T(A)$ nonempty and compact.
We say that $(D\subseteq A)$ has \emph{weak CPoU} if the sub-$C^\ast$-algebra $(\kappa(D^\cU) \subseteq A^\cU)$ is unital and if for any $\|\cdot\|_{2,T_\cU(A)}$-separable subset $S$ of $A^\cU$, for any $a_1,\dots,a_k \in D_+$ and $\delta > 0$ satisfying
\begin{equation}   
\sup_{\tau \in T(A)} \min\{\tau(a_1),\dots,\tau(a_k)\} < \delta,
\end{equation} 
and any positive contraction $q \in \kappa(D^\cU) \cap A'$ with $\tau(q)>0$ for all $\tau\in T_\cU(A)$,
there are positive contractions $e_1,\dots,e_k \in \kappa(D^\cU) \cap S'$ (not necessarily projections) such that
\begin{equation}
\sum_{i=1}^k e_i = 1_{A^\cU} \ \text{and}\ \tau(a_ie_iq)\leq \dl\tau(e_iq), \quad \tau \in T_\cU(A),\, i=1,\dots,k. 
\end{equation} 
\end{defn}

The concept of weak CPoU has appeared without a formal name in \cite[Lemma~3.6]{CETWW}, where $(A\subseteq A)$ is proven to have weak CPoU assuming that $A$ is nuclear and $T(A)$ is nonempty and compact, and then in \cite[Theorem~6.16]{CCEGSTW}, where this result is generalized, in particular dropping the nuclearity hypothesis, and presented with a more transparent proof. 

We state here \cite[Lemma 6.15]{CCEGSTW}, an auxiliary lemma used to prove weak CPOU (for $A\subseteq A$).
Our statement is in a less general form than in \cite{CCEGSTW}, as the version there is in the more abstract setting of a tracially complete $C^\ast$-algebra $(\mathcal M,X)$ (see \cite[\S 3.3]{CCEGSTW} for more details).

\begin{prop}[{\cite[Lemma~6.15]{CCEGSTW}}]\label{prop:wcpou}
Let $A$ be a separable $C^\ast$-algebra with $T(A)$ nonempty and compact. Let $(h_n)_{n=1}^\infty\subseteq A_+^1$ be an approximate unit of $A$, let $q\in A_+^1$ be such that $\tau(q)>0$ for all $\tau\in T(A)$, let $\epsilon>0$, let $n_0\in\mathbb{N}$, let $\mathcal{F} \subseteq A$ be a finite subset, and let $a_1,\dots,a_k\in A_+$ be such that
\begin{equation} \label{eq:q}
\min_{1\le i \le k}\tau(a_iq) <\delta\tau(q),\quad \tau\in T(A).
\end{equation}
Then, there are positive contractions $e_1,\dots,e_k \in A_+^1$ and $n\ge n_0$ such that
\begin{enumerate}
\item \label{item1:unit} $\|\sum_{i=1}^ke_i-h_n\|_{2, T(A)}<\epsilon$,
\item $\|[e_i,b]\|_{2, T(A)}<\epsilon$, for all $b\in\mathcal{F}$, $i=1,\dots, k$,
\item $\tau(a_ie_iq)<\delta\tau(e_iq)+\epsilon$, for all $\tau\in T(A)$, $i=1,\dots,k$.
\end{enumerate}
\end{prop}
\begin{proof}
This statement is an immediate consequence of \cite[Lemma~6.15]{CCEGSTW} applied to the tracial completion $(\overline{A}^{T(A)}, T(A))$ (see \cite[\S 3.3]{CCEGSTW} for the definition and more details). Note that item \ref{item1:unit} follows since every approximate unit in $A$ converges in $\| \cdot \|_{2,T(A)}$-norm to the unit of $\overline{A}^{T(A)}$ (see \cite[Proposition 3.9]{CCEGSTW}). Notice moreover that, while the statement of \cite[Lemma~6.15]{CCEGSTW} asks for
\begin{equation}
\min_{1\le i \le k}\tau(a_i) <\delta\tau(q),\quad \tau\in T(A),
\end{equation}
the proof actually only requires \eqref{eq:q}.
\end{proof}

The following is a generalization of \cite[Lemma 3.7]{CETWW}, and it shows that the gap between CPoU and weak CPoU is precisely uniform property $\Gamma$.

\begin{prop}
\label{prop:weakCPoUplusGamma}
Let $(D\subseteq A)$ be a sub-$C^\ast$-algebra with $A$ separable and $T(A)$ nonempty and compact.
If $(D\subseteq A)$ has both weak CPoU and property $\Gamma$, then it has CPoU.
\end{prop}

\begin{proof}
The proof follows closely the one of \cite[Lemma~3.7]{CETWW}, so we only sketch the main ideas.

Let $S\subseteq A^\cU$ be a $\| \cdot \|_{2,T_\cU(A)}$-separable subset that contains the canonical image of $A$ in $A^\cU$, let $a_1,\dots,a_k\in D_+$, and let $\delta > 0$ be such that
\begin{equation} 
\sup_{\tau \in T(A)} \min \{ \tau(a_1),\dots,\tau(a_k)  \} < \delta.
\end{equation} 
Consider the set $I$ of all $\alpha\in [0,1]$ such that there are pairwise orthogonal projections $p_1,\dots,p_k\in \kappa(D^\cU)\cap S'$ satisfying
\begin{equation} \label{eq:CPoUalpha} 
\tau\left( \sum_{i=1}^k p_i  \right) = \alpha \text{ and }
\tau(a_ip_i) \leq \delta \tau(p_i), \quad \tau\in T_\cU(A),\, i=1,\dots,k.
\end{equation}
Note that $I$ is nonempty since $0 \in I$.
The goal is to show that $\beta \coloneqq \sup I = 1$, which is enough because $I$ can be proved to be closed using Kirchberg's $\e$-test (see \cite[Lemma A.1]{Kir06}).

The claim that $\beta = 1$ follows from a maximality argument. That is, we assume $\beta  < 1$ and try to construct projections in the orthogonal remaining corner.
Then summing the two families would yield projections that take up more traces than $\beta$, which contradicts the assumption that $\beta$ is the supremum.

To make this work, 
let $p_1^{(1)}, \dots ,p_k^{(1)}$ be pairwise orthogonal projections  in $\kappa(D^\cU)\cap S'$ witnessing \eqref{eq:CPoUalpha} for $\beta$.
Set
\begin{equation} 
q\coloneqq 1_{A^\cU} -\sum_{i=1}^k p_i^{(1)} \in \kappa(D^\cU) \cap S'.
\end{equation} 
Note that $q \in A'$ because $A\subseteq S$.
Since $(D \subseteq A)$ has weak CPoU, we can find positive contractions $e_1^{(2)}, \dots , e_k^{(2)}$ in $\kappa(D^\cU) \cap (S\cup\{q\})'$ such that
\begin{equation}  
\sum_{i=1}^k e_i^{(2)} = 1_{A^\cU} \ \text{and}\ \tau(a_ie_i^{(2)}q)\leq \dl\tau(e_i^{(2)}q), \ \tau \in T_\cU(A), \,  i=1,\dots,k. 
\end{equation} 
Since $(D \subseteq A)$ is also assumed to satisfy uniform property $\Gamma$, by Lemma~\ref{lem:projectionization} there are projections
\begin{equation} 
p_1^{(2)},\dots,p_k^{(2)}\in \kappa(D^\cU) \cap (S \cup \{q\})'
\end{equation} 
satisfying 
\begin{equation}
\tau(p_i^{(2)}q) = \tau(e_i^{(2)}q) \text{ and } \tau(p_i^{(2)}qa_i)=\tau(e_i^{(2)}qa_i), \quad i =1, \dots, k.
\end{equation} 

Set
\begin{equation}
\begin{split}
S_1 &\coloneqq C^\ast(S \cup \{ q \} \cup \{a_i, p_i^{(1)}, p_i^{(2)} \}_{i =1}^k) \subseteq A^\cU, \quad\text{and} \\
T &\coloneqq C^\ast(\{q\} \cup \{ a_i, p_i^{(2)}\}_{i=1}^k) \subseteq \kappa(D^\cU).
\end{split}
\end{equation}

Apply Lemma~\ref{lemma:Gamma+} to obtain a partition of unity consisting of orthogonal projections $q_1,\dots,q_k\in \kappa(D^\cU) \cap S_1'$ such that
\begin{equation} 
 \tau(q_ia) = \frac1k \tau(a), \quad \tau \in T_\cU(A), \, a \in T, \, i = 1, \dots, k. 
\end{equation} 
Set
\begin{equation} p_i \coloneqq p_i^{(1)} + qp_i^{(2)}q_i \in \kappa(D^\cU) \cap S',\quad i=1,\dots,k.\end{equation} 

Then, as the calculations in \cite[Lemma~3.7]{CETWW} show, these are pairwise orthogonal projections satisfying
\begin{equation} 
\sum_{i=1}^k \tau(p_i)=\beta+\frac1k(1-\beta)>\beta \text{ and } \tau(a_ip_i)\leq \delta\tau(p_i),\ i = 1, \dots, k.
\end{equation} 
This contradicts maximality of $\beta$, as desired.
\end{proof}

\subsection{The small boundary property and CPoU}
\begin{prop}
\label{prop:CPoU_SBP}
Let $(D\subseteq A)$ be a unital sub-$C^\ast$-algebra with $D$ abelian, $A$ separable and $T(A)$ nonempty. If $(D\subseteq A)$ has CPoU, then $(D\subseteq A)$ has the small boundary property.
\end{prop}

\begin{proof}
Write $D=C(X)$, where the compact metric space $X$ denotes the spectrum of $D$. For a continuous function $f \in C(X)$ we denote the \emph{open support} of $f$ as
\begin{equation}  \suppo(f)\coloneqq f^{-1}(\mathbb C \setminus \{0\}).   \end{equation}  
For a trace $\tau\in T(A)$, the restriction $\tau\vert_{C(X)}$ induces a unique Borel probability measure on $X$ which we denote by $\mu_\tau$.

We shall verify the equivalent formulation of the small boundary property in Proposition~\ref{prop:sbp}.\ref{sbpitem3}, so let $\delta, \e > 0$ be given.
Fix $\tau \in T(A)$; by compactness, there is a finite Borel partition of $X$ consisting of sets of sufficiently small diameter and, by inner regularity of $\mu_\tau$, we obtain compact subsets of the sets in the partition so that the union of these has large enough measure $\mu_\tau$. Taking sufficiently small open neighborhoods of these compact sets and using Urysohn's lemma, it follows that there is a finite collection $\Omega_\tau\subseteq C(X)_+^1$ consisting of pairwise orthogonal positive continuous functions such that
\begin{enumerate}
	\item $\mathrm{diam}(\suppo f)  < \delta$ for all $f\in \Omega_\tau$, and
	\item $\tau\left( 1_{C(X)} - \sum_{f\in \Omega_\tau} f   \right)  < \e/3$.
\end{enumerate} 
Define 
\begin{equation} \label{eq:CPoUSBP.1}
g_\tau \coloneqq 1_{C(X)} - \sum_{f\in \Omega_\tau} f.
\end{equation}

 Since $T(A)$ is compact, there exist $\tau_1,\dots,\tau_k\in T(A)$ such that
\begin{equation} 
\sup_{\tau\in T(A)} \min\{ \tau(g_{\tau_1}),\dots,\tau(g_{\tau_k})  \} < \e/3.
\end{equation} 
Applying CPoU, we obtain pairwise orthogonal projections $p_1,\dots,p_k$ in $\kappa(C(X)^\cU)$ such that
\begin{equation} 
	  \sum_{i=1}^k p_i  = 1 \text{ and }
	 \tau(p_i g_{\tau_i} ) \leq \frac\e3 \tau(p_i), \ \tau\in T_\cU(A), \,i=1,\dots,k.
\end{equation} 

We may lift the projections $p_1,\dots, p_k$ to pairwise orthogonal elements in $\ell^\infty(\mathbb N, C(X))_+^1$, and by taking elements at a sufficiently late stage in this sequence, we obtain pairwise orthogonal positive contractions $e_1,\dots,e_k$ in $C(X)_+^1$ such that, for every $\tau\in T(A)$ and $i=1,\dots,k$
\begin{equation}
\label{eq:CPoUSBP.2}
1 - \frac{\e}{3} < \tau\left( \sum_{i=1}^k e_i    \right) \le 1,
\end{equation}
and
\begin{equation}
\label{eq:CPoUSBP.3}
\tau(e_i g_{\tau_i} ) \leq \frac\e3 \tau(e_i) + \frac{\e}{3k}.
\end{equation}

Consider the finite collection of pairwise orthogonal contractions in $C(X)$ given by
\begin{equation}
\Omega \coloneqq \{ e_if \colon i=1,\dots,k, \, f\in \Omega_{\tau_i}  \}\subseteq C(X)_+^1.
\end{equation}
For every $\tau\in T(A)$ one has
\begin{equation}
\begin{array}{rcl}
\hspace*{-2em} \tau\left(1_A - \sum_{i=1}^k \sum_{f\in \Omega_{\tau_i}} e_if  \right) 
&\stackrel{{\eqref{eq:CPoUSBP.2}}}{<}& \tau \left(  \sum_{i=1}^k e_i \left(  1_A - \sum_{f\in \Omega_{\tau_i}} f  \right)    \right) + \frac{\e}{3} \\
&\stackrel{{\eqref{eq:CPoUSBP.1}}}=& \tau\left( \sum_{i=1}^k e_ig_{\tau_i}   \right) + \frac{\e}{3} \\
&\stackrel{{\eqref{eq:CPoUSBP.3}}}\leq& \sum_{i=1}^k \frac\e3 \cdot \tau(e_i) + k\cdot \frac{\e}{3k} + \frac{\e}{3} \\
&\stackrel{{\eqref{eq:CPoUSBP.2}}}{\leq}& \frac\e3 + \frac{\e}{3} + \frac{\e}{3} \\ &=& \e.
\end{array} 
\end{equation} 
Now the finite collection of disjoint open sets
\begin{equation}
\mathcal{I} \coloneqq \{  \suppo(e_if) \colon i=1,\dots,k, \, f\in \Omega_{\tau_i}  \}
\end{equation}
consists of sets whose diameters are uniformly bounded by $\delta$, because each is a subset of some $\suppo(f)$ where $f$ is in some $\Omega_{\tau}$. Finally, for $\tau \in T(A)$, one has
\begin{equation} 
\mu_\tau\left( \bigsqcup_{U\in \mathcal{I}} U \right) \geq \tau \left(  \sum_{i=1}^k \sum_{f\in \Omega_{\tau_i}} e_if   \right) > 1- \e,
\end{equation} 
so the small boundary property of $(D\subseteq A)$ follows from Proposition~\ref{prop:sbp}.\ref{sbpitem3}.
\end{proof}

\section{Main result and consequences} \label{sec:Main}
In this section we prove Theorem~\ref{thm:GammaSBP} and derive from it Theorem~\ref{thm:AlmostFinite}. 
The following result establishes weak CPoU for all sub-$C^\ast$-algebras originating from an action of a discrete amenable group on an abelian $C^\ast$-algebra.
\begin{thm}
\label{thm:MainWeakCPoU}
Let $X$ be a compact metric space and let $\alpha \colon G \curvearrowright X$ be an action of a countable discrete amenable group $G$ on $X$.
Then the sub-$C^\ast$-algebra $(C(X)\subseteq C(X)\rtimes_\alpha G)$ has weak CPoU.
\end{thm}

\begin{proof} 
Set $D\coloneqq C(X)$, $A\coloneqq C(X) \rtimes_\alpha G$. We use $\alpha$ to also denote the canonically induced action $G \curvearrowright D$, and let $\{u_g\}_{g\in G} \subseteq A$ denote the unitary copy of $G$ in the crossed product that spatially implements $\alpha$, i.e.\ $\alpha_g(f)=u_gfu_g^*$ for all $f\in D$ and $g\in G$. For every $h\in D$ and nonempty finite subset $F \subseteq G$, we establish the notation
\begin{equation}
h^{(F)}\coloneqq \frac{1}{|F|}\sum_{g\in F}\alpha_g(h).
\end{equation}
Note that, whenever $h$ is positive or a contraction, then so is $h^{(F)}$.

To verify weak CPoU of $(D\subseteq A)$, first note that the sub-$C^\ast$-algebra $(\kappa(D^\cU)\subseteq A^\cU)$ is unital, since $1_A\in D$. Let $a_1,\dots,a_k\in D_+$ and $\delta>0$ be such that
\begin{equation} \label{eq:wcpou1}
\sup_{\tau \in T(A)} \min \{\tau(a_1), \dots, \tau(a_k) \} < \delta.
\end{equation}
Let $q\in\kappa(D^\cU)\cap A'$ be a positive contraction with $\tau(q)>0$ for all $\tau\in T_\cU(A)$, and let $S \subseteq A^\cU$ be a $\|\cdot\|_{2,T_\cU(A)}$-separable subset, with $\{s_\ell\}_{\ell = 1}^\infty \subseteq S$ a $\|\cdot\|_{2,T_\cU(A)}$-dense subset. For each $\ell\in\mathbb{N}$, let $(s_{\ell,m})_{m=1}^\infty\in\ell^\infty(\mathbb{N},A)$ be a representative sequence for $s_\ell$. 

Since $A$ is generated as a $C^\ast$-algebra by $D$ and $\{u_g\}_{g\in G}$, and since $D$ is abelian, for any $m\in\mathbb{N}$, there is $\eta_m>0$ and a finite set $K_m \subset G$ with the following property: for any contraction $f\in D$, if $\|[f, u_g]\|_{2, T(A)}<\eta_m$ for all $g\in K_m$, then $\|[f,s_{\ell,m}]\|_{2,T(A)}<1/m$ for all $\ell = 1,\dots,m$.

Fix a lift $(q_m)_{m=1}^\infty\in\ell^\infty(\mathbb{N},D)_+^1$ of $q$. Consider the set $\Omega\subseteq\mathbb{N}$ of all $m\in\mathbb{N}$ such that there exists a finite, symmetric set $H_m \subseteq G$ that is $(K_m,\eta_m)$-invariant (i.e.\ $\frac{|K_m H_m \triangle H_m|}{|H_m|} < \eta_m$) and such that
\begin{equation}
\min_{1\le i \le k}(a_iq_m)^{(H_m)}(x)<\delta q_m^{(H_m)}(x),\quad x \in X.
\end{equation}

\begin{claim}
The set $\Omega$ belongs to the ultrafilter $\cU$.
\end{claim}
\renewcommand\qedsymbol{$\blacksquare$}
\begin{proof}[Proof of the claim]
The proof follows a standard argument (see for instance \cite[Proposition~3.3]{KerrSzabo}).
Suppose that the claim is false, in order to reach a contradiction. Since $\cU$ is an ultrafilter, the complement $\mathbb{N}\setminus\Omega$ lies in $\cU$.

Fix $m\in\mathbb{N}\setminus \Omega$. Since $G$ is amenable, let $(F_n)_{n=1}^\infty$ be a F{\o}lner sequence of $G$ consisting of finite, symmetric sets that are all $(K_{m},\eta_{m})$-invariant. Since $m\not\in\Omega$, for each $n\in\mathbb{N}$, there exists a point $x_n\in X$ such that
\begin{equation}
\min_{1\le i \le k}(a_iq_{m})^{(F_n)}(x_n)\ge\delta q_m^{(F_n)}(x_n).
\end{equation}
Let $\sigma_n$ be the state on $D$ defined as
\begin{equation}
\sigma_n(f)\coloneqq \frac{1}{|F_n|}\sum_{g\in F_n}\alpha_g(f)(x_n),\quad f\in D
\end{equation}
and note that, for each $n\in\mathbb{N}$, we have
\begin{equation}\label{eq:intermediate1}
\min_{1\le i\le k}\sigma_n(a_iq_m)=\min_{1\le i \le k}(a_iq_m)^{(F_n)}(x_n)\ge\delta q_m^{(F_n)}(x_n)=\delta\sigma_n(q_m).
\end{equation}

Using weak$^*$ compactness of the state space $S(D)$ of $D$, by passing to a subsequence, we may assume that $(\sigma_n)_{n=1}^\infty$ converges to some $\tau_m \in S(D)$.
Then, by construction, $\tau_m$ is $G$-invariant, and therefore is in $T(A)\vert_D$, since it extends to $A$ as $\tau_m \circ E$, where $E \colon A \to D$ is the canonical conditional expectation. By \eqref{eq:intermediate1}, we obtain
\begin{equation}\label{eq:intermediate2}
\min_{1\le i \le k}\tau_m(a_iq_m)\ge\delta\tau_m(q_m).
\end{equation}
This provides a trace $\tau_m\in T(A)$ for each $m\in\mathbb N\setminus \Omega$.
Choose $\tau_m\in T(A)$ aribtrarily for $m\in \Omega$, and use the sequence $(\tau_m)$ to define a limit trace $\tau\in T_\cU(A)$.
Then, since $\mathbb{N}\setminus\Omega\in\cU$ by assumption, and by \eqref{eq:intermediate2}, we have that
\begin{equation}\label{eq:intermediate3}
\min_{1\le i\le k}\tau(a_iq)\ge\delta\tau(q).
\end{equation}

However, $q\in\kappa(D^\cU)\cap A'$ and $\tau(q)>0$ imply that the map $A\to\mathbb{C}$ given by $a\mapsto\tau(q)^{-1}\tau(a q)$ is in $T(A)$, and so \eqref{eq:intermediate3} violates \eqref{eq:wcpou1}, leading to a contradiction. 
This proves the claim.
\end{proof}

\renewcommand\qedsymbol{$\square$}

For each $m\in\Omega$, choose a finite, symmetric set $H_m \subset G$ that is $(K_m,\eta_m)$-invariant, and such that $\min_{1\le i\le k}(a_iq_m)^{(H_m)}(x)< \delta q_m^{(H_m)}(x)$ for all $x \in X$. Consider the sets
\begin{equation}\label{eq:opencover}
U_{i,m}\coloneqq\left\{x\in X:(a_iq_m)^{(H_m)}(x)<\delta q_m^{(H_m)}(x)\right\}\subseteq X,\quad i=1,\dots,k,
\end{equation}
which form an open cover of $X$, and let $f_{1,m},\dots, f_{k,m}\in D_+^1$ be a partition of unity for $X$ subordinate to the cover of \eqref{eq:opencover}. Set
\begin{equation}\label{eq:contractionsaveraged}
e_{i,m}\coloneqq f_{i,m}^{(H_m)}\in D_+^1,\quad i=1,\dots,k.
\end{equation}
For $m\in\mathbb{N}\setminus\Omega$ and $1\le i\le k$ set $e_{i,m}\coloneqq 0$ and  let $e_i\in\kappa(D^\cU)_+^1$ be the class of $(e_{i,m})_{m=1}^\infty$. 

It is clear that $\sum_{i=1}^ke_i=1_{A^\cU}$. To verify that each $e_i$ commutes with $S$, it suffices to verify commutation with each $s_\ell$. Fix $\ell\in\mathbb{N}$. Note that for any $m\in\Omega$ with $m\ge \ell$ and any $g\in K_m$, we have that 
\begin{equation}
\begin{array}{rcl}
 \|[e_{i,m},u_g]\| 
&\stackrel{\eqref{eq:contractionsaveraged}}{\le}& \frac{1}{|H_m|}\cdot \left\|\sum_{t\in gH_m}\alpha_t(f_{i,m})-\sum_{r\in H_m}\alpha_r(f_{i,m})\right\| \\
&\le& \frac{|K_mH_m\triangle H_m|}{|H_m|}\\ &<& \eta_m.
\end{array}
\end{equation}
As $\|[e_{i,m},u_g]\|_{2,T(A)}  \le \|[e_{i,m},u_g]\| < \eta_m$, we therefore get\
\begin{equation}
\|[e_{i,m},s_{\ell,m}]\|_{2,T(A)}<1/m.
\end{equation}
We conclude that each $e_i$ commutes with $s_\ell$.

Since traces in $A$ restrict to invariant measures on the spectrum of $D$, given $a,b \in D$, $g \in G$ and $\tau \in T(A)$, we have that $\tau(a\alpha_g(b)) = \tau(\alpha_{g^{-1}}(a)b)$, so if $H \subseteq G$ is finite and symmetric, this implies in turn $\tau(ab^{(H)}) = \tau(a^{(H)}b)$. We can use this to see that, for any $\tau\in T(A)$, any $m\in\Omega$ and any $i=1,\dots,k$, the following holds
\begin{equation}
\begin{array}{rcl}
\tau(a_ie_{i,m}q_m) &\stackrel{\eqref{eq:contractionsaveraged}}{=}& \tau\left(a_i q_m\cdot f_{i,m}^{(H_m)}\right) \\
& =  & \tau \left( (a_i q_m)^{(H_m)}f_{i,m}\right) \\
& \stackrel{\eqref{eq:opencover}}{\le} & \delta \tau\left( q_m^{(H_m)}f_{i,m}\right) \\
& = & \delta \tau\left(q_m f_{i,m}^{(H_m)}\right) \\
& \stackrel{\eqref{eq:contractionsaveraged}}{=} &\delta \tau(q_me_{i,m}).
\end{array} 
\end{equation} 
and, since $\Omega\in\cU$, this shows that $\tau(a_ie_iq)\le\delta\tau(e_iq)$ for all $\tau\in T_\cU(A)$, and all $i=1,\dots,k$.
\end{proof}

In the following result we show that weak CPoU holds automatically also in the case of Cartan pairs (see \cite[Definition~5.1]{Renault:IMSB}). Before stating the theorem we recall some definitions. Given a sub-$C^\ast$-algebra $(D \subseteq A)$, the set of \emph{normalizers} is defined as
\begin{equation} \label{eq:normalizers}
\mathcal{N}_A(D) \coloneqq \{ a \in A \colon a^* D a + aDa^* \subseteq D \}
\end{equation}
We say that $(D \subseteq A)$ is \emph{regular} if the $C^\ast$-algebra generated by $\mathcal{N}_A(D)$ is $A$ and we say that  $(D \subseteq A)$
has the \emph{almost extension property} if there is a weak$^*$-dense set of pure states of $D$ that has unique extension to $A$. A sub-$C^\ast$-algebra
$(D \subseteq A)$ is a \emph{Cartan pair} if $D$ is a maximal abelian subalgebra, $(D \subseteq A)$ is regular, there exists a faithful conditional expectation $E \colon A \to D$, and $D$ is non-degenerate, i.e., it contains an approximate unit of $A$.

\begin{thm}\label{thm:weak-cpou-cartan-pairs}
Let $(D\subseteq A)$ be a non-degenerate sub-$C^\ast$-algebra such that $A$ is separable and $T(A)$ is nonempty and compact. Suppose moreover that
\begin{enumerate}
\item $D$ is abelian,
\item $(D\subseteq A)$ is regular,
\item $(D\subseteq A)$ has the almost extension property,
\item \label{weak-cpou:item4} there is a conditional expectation $E\colon A\to D$.
\end{enumerate}
Then $(D\subseteq A)$ has weak CPoU. In particular, if $(D\subseteq A)$ is a Cartan pair, then $(D\subseteq A)$ has weak CPoU.
\end{thm}

\begin{proof}

The sub-$C^\ast$-algebra $\kappa(D^\cU)\subseteq A^\cU$ is unital, since $(D\subseteq A)$ is non-degenerate. To verify that $(D \subseteq A)$ has weak CPoU, let $a_1,\dots,a_k\in D_+$, and let $\delta>0$ be such that
\begin{equation}\label{eq:wcpoudiag}
\sup_{\tau \in T(A)} \min \{\tau(a_1), \dots, \tau(a_k) \} < \delta.
\end{equation}
Let $S\subseteq A^\cU$ be a $\|\cdot\|_{2,T_\cU(A)}$-separable subset, and let $q\in\kappa(D^\cU)\cap A'$ be a positive contraction with $\tau(q)>0$ for all $\tau\in T_\cU(A)$. Note that for any $\tau\in T_\cU(A)$, since $q$ commutes with $A$ and $\tau(q)>0$, the map $A\to\mathbb{C}$ given by $a\mapsto\tau(q)^{-1}\tau(aq)$ is in $T(A)$ and so \eqref{eq:wcpoudiag} yields that
\begin{equation}\label{eq:wcpoulimit}
\min_{1\le i\le k}\tau(a_iq)<\delta\tau(q), \quad\tau\in T_\cU(A).
\end{equation} 

Let $(q_m)_{m=1}^\infty\in\ell^\infty(\mathbb{N},D)_+^1$ be a lift of $q$. By \eqref{eq:wcpoulimit} and since $\tau(q)>0$ for all $\tau\in T_\cU(A)$, it follows that the set $\Omega\subseteq\mathbb{N}$ of all $m\in\mathbb{N}$ satisfying $\tau(q_m)>0$ and $\min_{1\le i\le k}\tau(a_iq_m)<\delta\tau(q_m)$ for all $\tau\in T(A)$ lies in $\cU$. By non-degeneracy, let $(h_n)_{n=1}^\infty\subseteq D$ be an approximate unit of $A$. Replacing $q_m$ with $h_n$ for $n$ big enough for all $m\in \mathbb{N}\setminus \Omega$, we may assume without loss of generality that $\tau(q_m)>0$ and $\min_{1\le i\le k}\tau(a_iq_m)<\delta\tau(q_m)$ for all $\tau\in T(A)$ and all $m\in\mathbb{N}$.

Let $\{s_\ell\}_{\ell =1}^\infty\subseteq S$ be a $\|\cdot\|_{2,T_\cU(A)}$-dense subset of $S$ and, for each $\ell\in\mathbb{N}$, let $(s_{\ell,m})_{m=1}^\infty\in\ell^\infty(\mathbb{N},A)$ be a representative sequence for $s_\ell$.

Fix $m\in\mathbb{N}$. Since $(D\subseteq A)$ is regular, there is a finite set $\mathcal{F}_m\subseteq\mathcal{N}_A(D)$ and $\eta_m>0$ with the property that, if $a\in A$ is a contraction such that $\|[a,b]\|_{2,T(A)}<\eta_m$ for all $b\in\mathcal{F}_m$, then $\|[a,s_{\ell,m}]\|_{2,T(A)}<1/m$ for all $\ell =1,\dots,m$. We now apply Proposition~\ref{prop:wcpou} with $\epsilon\coloneqq \min\{\frac{\eta_m^2}{2\max_{b\in \mathcal{F}_m}\|b\|+1},\frac{1}{m}\}$ therein and obtain contractions $f_{1,m},\dots,f_{k,m}\in A_+^1$ and some $n_m\ge m$ such that
\begin{equation}\label{eq:sumunit}
\left\|\sum_{i=1}^kf_{i,m}-h_{n_m}\right\|_{2,T(A)}<1/m,
\end{equation}
\begin{equation}\label{eq:approxcom}
\|[f_{i,m},b]\|_{2,T(A)}<\epsilon, \quad b\in \mathcal{F}_m\cup\mathcal{F}_m^*,\; i=1,\dots,k,
\end{equation}
and
\begin{equation}\label{eq:traceineq}
\tau(a_if_{i,m}q_m)<\delta\cdot \tau(f_{i,m}q_m)+1/m, \quad \tau\in T(A), \; i=1,\dots,k.
\end{equation}
For $i=1,\dots,k$ we set $e_{i,m}\coloneqq E(f_{i,m})\in D_+^1$ and we let $e_i\in\kappa(D^\cU)_+^1$ be the class of $(e_{i,m})_{m=1}^\infty\in\ell^\infty(\mathbb{N},D)$.

Note that $\tau \circ E \in T(A)$ for all $\tau \in T(A)$ (see for instance \cite[Corollary 3.3]{crytser2017traces}). Together with the Cauchy--Schwarz inequality this entails that
\begin{equation}
\begin{array}{rcl}

\|E(a)\|_{2,T(A)} &=& \sup_{\tau \in T(A)} \tau(E(a)E(a^*))^{1/2} \\
&\le & \sup_{\tau \in T(A)} \tau(E(aa^*))^{1/2} \le  \| a \|_{2, T(A)}.
\end{array}
\end{equation}

To verify that $\sum_{i=1}^me_i=1_{A^\cU}$, we observe that for any $m\in\mathbb{N}$ we have
\begin{equation}
\begin{array}{rcl}
\left\|\sum_{i=1}^ke_{i,m}-h_{n_m}\right\|_{2,T(A)} &=& \left\|E\left(\sum_{i=1}^kf_{i,m}-h_{n_m}\right)\right\|_{2,T(A)} \\
 & \le &  \left\|\sum_{i=1}^kf_{i,m}-h_{n_m}\right\|_{2,T(A)} \\ 
 &\stackrel{\eqref{eq:sumunit}}{<}& 1/m,
\end{array}
\end{equation}
and that $(h_{n_m})_{m=1}^\infty$ is a representative of $1_{A^\cU}$ (\cite[Proposition~1.11]{CETWW}). To see that each $e_i$ commutes with $S$, it suffices to verify that it commutes with $s_\ell$ for any $\ell\in\mathbb{N}$, so let $\ell\in\mathbb{N}$. For $i=1,\dots,k$, $m\ge \ell$, $b\in\mathcal{F}_m$ and $\tau\in T(A)$, it follows by the triangle inequality that the value $\tau([e_{i,m},b]^*[e_{i,m},b])$ is smaller than the sum of the following two quantites
\begin{equation}
\begin{split}
C_0&\coloneqq \left|\tau\left(b^*E(f_{i,m})^2b-b^*E(f_{i,m})bE(f_{i,m})\right)\right|, \;\text{and} \\
C_1&\coloneqq\left|\tau\left(E(f_{i,m})b^*E(f_{i,m})b -E(f_{i,m})b^*bE(f_{i,m})\right)\right|.
\end{split}
\end{equation}
As shown in \cite[Lemma~3.2]{crytser2017traces} (see also \cite[Lemma~6$^\mathrm{o}$]{Kumjian:CJM}), the almost extension property ensures that
\begin{equation}\label{eq:entangling}
vE(a)v^* = E(vav^*),\quad a\in A,\; v\in\mathcal{N}_A(D).
\end{equation}
Using that $D$ is in the multiplicative domain of $E$ and since $E(f_{i,m})\in D$, we now have
\begin{equation}\label{eq:rev1}
\begin{split}
b^*E(f_{i,m})^2b&=b^*E(E(f_{i,m})f_{i,m})b \\
& \stackrel{\mathmakebox[\widthof{=}]{\eqref{eq:entangling}}}{=}\; E(b^*E(f_{i,m})f_{i,m}b),
\end{split} \end{equation}
and likewise since $b^*E(f_{i,m})b\in D$,
\begin{equation}\label{eq:rev2}
b^*E(f_{i,m})bE(f_{i,m}) = E(b^*E(f_{i,m})\cdot bf_{i,m}).
\end{equation}
Inspecting the definition of $C_0$, \eqref{eq:rev1} and \eqref{eq:rev2} imply that
\begin{equation}\label{eq:rev3}
C_0 = | (\tau\circ E)(b^* E(f_{i,m}) \cdot [f_{i,m},b])|.
\end{equation}
Using the Cauchy--Schwarz inequality together with \eqref{eq:rev3} in the first line, that $E(f_{i,m})$ is a contraction on the second line, and our earlier observation that $\tau\circ E\in T(A)$ in the fourth line, we have 
\begin{equation}\label{eq:calculation2rev}
\begin{array}{rcl}
C_0 & \le & (\tau\circ E)(b^*E(f_{i,m})^2 b)^{1/2} \cdot (\tau\circ E)([f_{i,m},b]^*[f_{i,m},b])^{1/2}\\
&\le & (\tau\circ E)(b^*b)^{1/2}\cdot (\tau\circ E)([f_{i,m},b]^*[f_{i,m},b])^{1/2}\\
& =  & \|b\|_{2,\tau\circ E} \cdot \|[f_{i,m},b]\|_{2,\tau\circ E} \\
&\le & \|b\| \cdot \|[f_{i,m},b]\|_{2,T(A)} \\
&\stackrel{\eqref{eq:approxcom}}{<}& \eta_m^2/2
\end{array}
\end{equation}
Using that $\tau$ is a trace, by the definition of $C_1$ we have
\begin{equation} 
C_1 = \left|\tau\left(bE(f_{i,m})b^*E(f_{i,m}) - bE(f_{i,m})^2b^*\right)\right|, 
 \end{equation}
so the above argument with $b^*$ in place of $b$ yields
\begin{equation}\label{eq:calculation3}
C_1< \eta_m^2/2,
\end{equation}
and so \eqref{eq:calculation2rev}, \eqref{eq:calculation3} yield 
\begin{equation}
\tau([e_{i,m},b]^*[e_{i,m},b])< \eta_m^2.
\end{equation}
Since $\tau\in T(A)$ is arbitrary, this implies that $\|[e_{i,m},b]\|_{2,T(A)}<\eta_m$ for all $b\in \mathcal{F}_m$.  Consequently, $\|[e_{i,m}, s_{\ell,m}]\|_{2,T(A)}<1/m$. As this is true for all $m\ge \ell$, we conclude that $e_i \in \kappa(D^\cU) \cap S'$.

Finally, for any $i=1,\dots,k$, $\tau\in T(A)$ and $m\in\mathbb{N}$, using that $D$ is a subset of the multiplicative domain of $E$ in the first and third lines and that $\tau\circ E\in T(A)$ in the second line, we see that
\begin{equation}
\begin{array}{rcl}
\tau(a_ie_{i,m}q_m) &=& (\tau\circ E)(a_if_{i,m}q_m) \\
& \stackrel{\eqref{eq:traceineq}}{<}& \delta\cdot (\tau\circ E)(f_{i,m}q_m)+1/m \\
&= &\delta\cdot \tau(e_{i,m}q_m)+1/m,
\end{array}
\end{equation}
which proves that $\tau(a_ie_iq)\le\delta \tau(e_iq)$ for all $\tau\in T_\cU(A)$ and all $i=1,\dots,k$.

For the last statement of the theorem, note that Cartan pairs have the almost extension property by \cite[Lemma 4.9]{CartanEtale}.
\end{proof}

Note that the combination of Proposition~\ref{prop:weakCPoUplusGamma} together with Theorem~\ref{thm:MainWeakCPoU} and Theorem~\ref{thm:weak-cpou-cartan-pairs} gives the following analogue of \cite[Lemma~3.7]{CETWW}.

\begin{cor} \label{cor:GammaCPoU}
Let $(D\subseteq A)$ be sub-$C^\ast$-algebra such that $A$ is separable. Suppose that either
\begin{enumerate}

\item $A=D\rtimes_\alpha G$ for an action $\alpha \colon G \curvearrowright D$ of a countable discrete amenable group $G$ on a unital abelian $C^\ast$-algebra $D$,

\item $(D\subseteq A)$ satisfies the assumptions of Theorem~\ref{thm:weak-cpou-cartan-pairs}, for instance
it is a Cartan pair where $A$ has nonempty trace space.
\end{enumerate}
If $(D\subseteq A)$ has uniform property $\Gamma$, then it has CPoU.
\end{cor}

\begin{thm}
\label{thm:Main2}
Let $X$ be a compact metric space, let $G$ be a countably infinite discrete amenable group and let $\alpha \colon G \curvearrowright X$ be a free action.
The following are equivalent:
\begin{enumerate}
\item \label{it:Main.1}
$\alpha$ has the small boundary property.
\item \label{it:Main.2}
$(C(X)\subseteq C(X)\rtimes_\alpha G)$ has uniform property $\Gamma$.
\item \label{it:Main.3}
$(C(X)\subseteq C(X)\rtimes_\alpha G)$ has CPoU.
\end{enumerate}
\end{thm}

\begin{proof} 
The implication \ref{it:Main.1}$\Rightarrow$\ref{it:Main.2} is \cite[Theorem~9.4, Remark~9.6]{KerrSzabo}. The implication \ref{it:Main.2}$\Rightarrow$\ref{it:Main.3} follows by Corollary~\ref{cor:GammaCPoU}, and \ref{it:Main.3}$\Rightarrow$\ref{it:Main.1} follows by Proposition~\ref{prop:CPoU_SBP} along with Remark~\ref{rmk:sbp of crossed product implies usual sbp}.
\end{proof}

We now use Theorem~\ref{thm:Main2} to show that Kerr's almost finiteness \cite[Definition~8.2]{Kerr:JEMS} is equivalent to tracial $\mathcal{Z}$-stability of the sub-$C^\ast$-algebra
originating from the action. The latter definition was originally introduced in \cite{LiaoTikuisis}, and we briefly recall it here (we refer to \cite{LiaoTikuisis} for further details).

Given a sub-$C^\ast$-algebra $(D \subseteq A)$, along with the set of {normalizers} defined in \eqref{eq:normalizers}, we consider the \emph{r-normalizers}, namely
\begin{equation} 
\mathcal{R} \mathcal{N}_A(D)  \coloneqq \{ a \in A \colon a^* D a  \subseteq D \}.
\end{equation} 
For $a,b \in D_+$, write $a \preceq_{(D \subseteq A)} b$ if there is a sequence $(t_n)_{n=1}^\infty$ in $\mathcal{RN}_A(D)$ such that $\lim_{n\to\infty} \|t_n^*bt_n - a\|=0$. We
write $a \sim_{(D \subseteq A)} b$ if $a  \preceq_{(D \subseteq A)} b$ and $b  \preceq_{(D \subseteq A)} a$. In what follows, for $n\in\mathbb{N}$,  $(D_n \subseteq M_n)$ denotes the sub-$C^\ast$-algebra of $n \times n$ matrices $M_n$ and diagonal $n \times n$ matrices $D_n$.

\begin{defn}[{\cite[Definition~3.2]{LiaoTikuisis}}] \label{def:trZstab}
Let $(D \subseteq A)$ be a unital sub-$C^\ast$-algebra. We say that $(D \subseteq A)$ is \emph{tracially $\mathcal Z$-stable} if $D \neq \mathbb{C} 1_A$ and if
for every $n \in \N$, every $\e>0$, every finite set $\mathcal{F} \subset A$, and every $s \in D_+ \setminus \{0\}$, there exists a completely positive contractive (c.p.c.) order zero map $\varphi \colon M_n \to A$ such that
\begin{enumerate}
\item \label{item:trZstab.1} $\varphi(\mathcal N_{M_n}(D_n)) \subseteq \mathcal N_A(D)$,
\item $1_A - \varphi(1_{M_n}) \preceq_{(D \subseteq A)} s$,\footnote{By \ref{item:trZstab.1}, it follows that $\varphi(1_{M_n}) \in D$, so this makes sense.}
\item $\|[a,\varphi(x)]\| < \e$ for all $a \in \mathcal{F}$ and every contraction $x \in M_n$.
\end{enumerate}
\end{defn}
Even though the requirement $D \neq \mathbb{C} 1_A$ does not appear in \cite[Definition~3.2]{LiaoTikuisis}, we choose to include it here, as it is in general necessary to avoid trivial situations where the zero map could witness tracial $\mathcal{Z}$-stability, for instance for $(\mathbb{C} \subseteq \mathbb{C})$ (see also \cite[Definition~2.1]{HirshbergOrovitz}, where $A \not \cong \mathbb{C}$ is required).

The following characterization of tracial $\mathcal{Z}$-stability relies on the notions of almost finiteness and dynamical comparison. We omit their full definition and refer the reader
to \cite{Kerr:JEMS, KerrSzabo} for further details, as these properties do not explicitly appear in our argument. We also assume that, for the following proof (and for the proof of Corollary~\ref{cor:AlmostFinite_product}), the reader is familiar with the basics of the theory of Cuntz comparison (see for instance \cite{ThielCuntzsemigrp}).

\begin{cor}\label{cor:AlmostFinite}
Let $X$ be a compact metric space, and let $\alpha \colon G \curvearrowright X$ be a free minimal action of a countably infinite discrete amenable group $G$.
The following are equivalent:
\begin{enumerate}
\item\label{it:AlmostFinite.1}
 The action $\alpha$ is almost finite (see \cite[Definition~8.2]{Kerr:JEMS}).
\item\label{it:AlmostFinite.2}
$(C(X)\subseteq C(X)\rtimes_\alpha G)$ is tracially $\mathcal{Z}$-stable.
\item\label{it:AlmostFinite.3}
$(C(X)\subseteq C(X)\rtimes_\alpha G)$ has uniform property $\Gamma$ and $\alpha$ has dynamical comparison (see \cite[Definition~3.2]{Kerr:JEMS}).
\item\label{it:AlmostFinite.4}
$\alpha$ has the small boundary property and dynamical comparison.
\end{enumerate}
\end{cor}

\begin{proof}
The implication \ref{it:AlmostFinite.1}$\Rightarrow$\ref{it:AlmostFinite.2} is \cite[Theorem~A (i)$\Rightarrow$(ii)]{LiaoTikuisis} (largely using the proof of \cite[Theorem~12.4]{Kerr:JEMS}).
\ref{it:AlmostFinite.3}$\Rightarrow$\ref{it:AlmostFinite.4} follows from Theorem~\ref{thm:Main2} and \ref{it:AlmostFinite.4}$\Rightarrow$\ref{it:AlmostFinite.1} is \cite[Theorem~6.1]{KerrSzabo}.

\ref{it:AlmostFinite.2}$\Rightarrow$\ref{it:AlmostFinite.3}:
Assume that $(C(X)\subseteq C(X)\rtimes_\alpha G)$ is tracially $\mathcal Z$-stable.
Dynamical comparison follows by \cite[Theorem~A (ii)$\Rightarrow$(iii)]{LiaoTikuisis}.
Let us prove that $(C(X)\subseteq C(X)\rtimes_\alpha G)$ has uniform property $\Gamma$.

Set $D\coloneqq C(X)$ and $A\coloneqq C(X)\rtimes_\alpha G$. Fix $k \in \mathbb N$ and identify $T(A)\vert_{D}$ with $M_G(X)$. Since $\alpha$ is free and $G$ infinite, every $\mu\in M_G(X)$ is atomless. By \cite[Proposition~3.4]{KerrSzabo} we can thus find a sequence $(s_n)_{n=1}^\infty$ of nonzero elements in $D_+$ with $\suppo(s_{n+1})\subseteq\suppo(s_n)$ for all $n\ge1$ and such that
\begin{equation} 
\sup_{\mu \in M_G(X)} \mu(\suppo(s_n)) \to 0.
\end{equation} 

By tracial $\mathcal Z$-stability of $(D\subseteq A)$, we can find c.p.c.\ order zero maps $\varphi_n:M_k \to A$ satisfying:
\begin{enumerate}[label=(\alph*)]
\item $\varphi_n(e_{ii})\in D_+$ for $i=1,\dots,k$,
\item \label{item:trZstab.b} $1_{A}-\varphi_n(1_{M_k}) \preceq_{(D\subseteq A)} s_n$,
\item the image of the induced map $\varphi=(\varphi_n)_{n=1}^\infty \colon M_k \to A^\cU$ commutes with the canonical image of $A$.
\end{enumerate}

Notice that, if $f_1,f_2 \in D_+$ are such that $f_1 \preceq_{(D\subseteq A)} f_2$, then clearly $f_1$ is Cuntz-below $f_2$ in $A$, and therefore $d_\tau(f_1)\le d_\tau(f_2)$ for every $\tau \in T(A)$, where $d_\tau$ is the \emph{dimension function}
\begin{equation} \label{eq:dimension}
d_\tau(f) \coloneqq \lim_{n \to \infty} \tau(f^{1/n}),\quad f\in D.
\end{equation}

For $\tau\in T(A)$, let $\mu_\tau \in M_G(X)$ be the measure corresponding to $\tau\vert_{D}$. Then the value in \eqref{eq:dimension} is precisely $\mu_\tau(\suppo(f))$, therefore using \ref{item:trZstab.b}, and since $1_{A^\cU}-\varphi(1_{M_k})$ is a positive contraction, we have that, for any limit trace $\tau\in T_\cU(A)$ induced by a sequence $(\tau_n)_{n=1}^\infty\subseteq T(A)$,
\begin{equation} 
\begin{array}{rcl}
\tau(|1_{A^\cU}-\varphi(1_{M_k})|^2) & = & \lim_{n\to\cU} \tau_n((1_{A}-\varphi_n(1_{M_k}))^2) \\
& \leq & \lim_{n\to\cU}\mu_{\tau_n}(\suppo(1_{A}-\varphi_n(1_{M_k}))) \\
& \leq & \lim_{n\to\cU}\mu_{\tau_n}(\supp(s_n))=0,
\end{array}
\end{equation} 
and therefore $\varphi(1_{M_k})=1_{A^{\cU}}$.

Since $\varphi$ is a unital c.p.c.\ order zero map, it is a $^\ast$-homomorphism, by \cite[Theorem~2.3]{WinterZacharias}. Moreover, we have $\varphi(e_{ii}) \in \kappa(D^\cU)$, and these form a partition of unity composed of projections.
As in the proof of \cite[Proposition~2.3]{CETWW} (using uniqueness of the trace on $M_k$), we obtain
\begin{equation}
\tau(\varphi(e_{ii})a)=\frac1k\tau(a),\quad \tau \in T_{\cU}(A), \, a \in D, \, i=1, \dots,k.\qedhere
\end{equation} 
\end{proof}

\subsection{Products of actions} \label{ss.cor}
We close the paper with the proof of Corollary~\ref{cor:SBPproducts} from the introduction, which we split in Corollaries \ref{cor:Gamma_product} and \ref{cor:AlmostFinite_product}.

The part concerning the small boundary property follows from a more general statement on uniform property $\Gamma$, obtained with a standard argument, which we state below. For the rest of this section we use $\otimes$ to denote the \emph{minimal} tensor product.

\begin{prop} \label{prop:Gamma_tensor}
Let $(D_A \subseteq A)$ and $(D_B \subseteq B)$ be unital sub-$C^\ast$-algebras such that  $A, B$ are both separable with nonempty trace spaces.
If $(D_A \subseteq A)$ has uniform property $\Gamma$ then $(D_A \otimes D_B \subseteq A \otimes B)$ has uniform property $\Gamma$.
\end{prop}
\begin{proof}
Let $\kappa\colon D_A^\cU\to A^\cU$ and $\lambda\colon (D_A\otimes D_B)^\cU\to (A\otimes B)^\cU$ be the canonical maps as in \eqref{eq:kappa}.
By assumption, given $k \in \mathbb{N}$, there is a partition of unity of $A^\cU$ consisting of projections $q'_1, \dots, q'_k \in \kappa(D_A^\cU) \cap A'$ such that
\begin{equation} \label{eq:GammaA}
\tau(q'_ia) = \frac{1}{k} \tau(a), \quad \tau \in T_\cU(A), \, a \in D_A, i = 1, \, \dots, k.
\end{equation}
For each $i=1,\dots, k$, let $(q'_{i,n})_{n=1}^\infty \in \ell^\infty(D_A)$ be a representing sequence of $q'_i$ and define $q_i \in \lambda((D_A\otimes D_B)^\cU)$ as the class of $(q'_{i,n} \otimes 1_B)_{n=1}^\infty$. Note that, by \cite[Proposition~3.5]{BBSTWW} and the Krein--Milman theorem, $T(A\otimes B)$ is equal to the closed convex hull of the set of traces
\begin{equation}
\{ \tau \otimes \sigma : \tau \in T(A), \, \sigma \in T(B) \},
\end{equation}
hence we have $\| a \otimes 1_B \|_{2, T(A \otimes B) } = \| a \|_{2,T(A)}$ for all $a \in A$. This is enough to conclude that $q_1, \dots, q_k$ form a partition of unity of $(A \otimes B)^\cU$ consisting of projections, and that they all belong to $\lambda((D_A \otimes D_B)^\cU)\cap (A \otimes B)'$.

To conclude, let us verify \eqref{eq:Gamma} from Definition \ref{def:Gamma}, hence fix $\tau \in T_\cU(A \otimes B)$, $c \in D_A \otimes D_B$ and $1 \le i \le k$. Without loss of generality, we can assume that $c = d \otimes e$, with $d \in D_A$  and $e\in (D_B)_+$. If $\tau(1_A \otimes e) = 0$ then
\begin{equation}
 \tau(q_i d \otimes e) = \frac{1}{k} \tau(d \otimes e) = 0.
\end{equation}
Otherwise, let $\Psi \colon A^\cU \to (A \otimes B)^\cU$ be the map sending the class of $(a_n)_{n=1}^\infty$ in $A^\cU$ to the class of $(a_n \otimes e)_{n=1}^\infty$ in $(A\otimes B)^\cU$.\footnote{This map is well-defined since $\|a\otimes e\|_{2,T(A\otimes B)}\le \|a\|_{2,T(A)}\cdot\|e\|$ for all $a\in A$.} The functional $\frac{\tau( \Psi(\cdot))}{\tau(1_A\otimes e)}$ defines a limit trace in $T_\cU(A)$, hence we get
\begin{equation}\begin{array}{rcl}
\tau(q_i (d \otimes e)) &=& \tau(1_A \otimes e) \displaystyle\frac{\tau(\Psi(q'_i d))}{\tau(1_A \otimes e)} \\
&\stackrel{\eqref{eq:GammaA}}{=}& \displaystyle\frac{\tau(1_A \otimes e)}{k} \cdot \frac{\tau(\Psi(d))}{\tau(1_A \otimes e)} \\ &=&
\frac{1}{k} \tau(d \otimes e),
\end{array} 
\end{equation}
as desired.
\end{proof}

\begin{cor} \label{cor:Gamma_product}
Let $X,Y$ be compact metric spaces, let $G,H$ be countably infinite discrete amenable groups and let $\alpha\colon G \curvearrowright X$, $\beta\colon G\curvearrowright Y$ be free actions. If $\alpha$ has the small boundary property, then so does the product action $\alpha \times \beta \colon G \times H \curvearrowright X \times Y$.
\end{cor}
\begin{proof}
We have $C(X\times Y)\rtimes_{\alpha \times \beta} (G\times H)\cong (C(X)\rtimes_\alpha G) \otimes (C(Y)\rtimes_\beta H)$, mapping $C(X\times Y)$ to $C(X)\otimes C(Y)$ in the natural way.

Based on this and Proposition~\ref{prop:Gamma_tensor}, uniform property $\Gamma$ for $(C(X)\subseteq C(X)\rtimes_\alpha G)$ implies uniform property $\Gamma$ for $(C(X\times Y)\subseteq C(X\times Y)\rtimes_{\alpha \times \beta} (G\times H))$.
The result then follows from Theorem~\ref{thm:Main2}.
\end{proof}

\begin{cor} \label{cor:AlmostFinite_product}
Let $X,Y$ be compact metric spaces, let $G,H$ be countably infinite discrete amenable groups and let $\alpha\colon G \curvearrowright X,\beta\colon G\curvearrowright Y$ be free minimal actions. If $\alpha$ is almost finite, then $\alpha \times \beta \colon G \times H \curvearrowright X \times Y$ is almost finite.
\end{cor}
\begin{proof}
Set  $(D_A\subseteq A)\coloneqq (C(X)\subseteq C(X)\rtimes_\alpha G)$ and $(D_B\subseteq B)\coloneqq (C(Y)\subseteq C(Y)\rtimes_\beta H)$. By Corollary~\ref{cor:AlmostFinite} it is sufficient to show that $(D_A \otimes D_B \subseteq A \otimes B)$ is tracially $\mathcal{Z}$-stable assuming that $(D_A \subseteq A)$ is.

To that end, let $\epsilon>0$, $n\in\mathbb{N}$, $\mathcal{F} \subseteq A\otimes B$ finite and $s\in(D_A\otimes D_B)_+$ nonzero be given.
Without loss of generality, we can assume that $\mathcal{F}$ consists of elementary tensors, say $\mathcal{F}=\{a_i\otimes b_i\}_{i=1}^k$ where $a_i\in A$ and $b_i\in B$ for all $i=1,\dots,k$. Moreover, by Kirchberg's slice lemma \cite[Lemma~4.1.9]{RordamStormer}, there is some $z\in D_A\otimes D_B$ such that $z^*z\in\overline{s(D_A\otimes D_B)s}$ and $zz^*$ is an elementary tensor $s_A\otimes s_B$ where $s_A\in (D_A)_+\setminus\{0\}$ and $s_B\in (D_B)_+\setminus\{0\}$. This in particular gives $s_A \otimes s_B
\preceq_{(D_A \otimes D_B \subseteq A \otimes B)} s$ (see e.g.\ \cite[Proposition 2.7]{ThielCuntzsemigrp}), so we can substitute $s$ with $s_A \otimes s_B$.
\renewcommand\qedsymbol{$\blacksquare$}
\begin{claim} \label{claim:subequivalence}
There exists a nonzero $t_A \in (D_A)_+$ such that 
\begin{equation}
t_A \otimes 1_B  \preceq_{(D_A \otimes D_B \subseteq A \otimes B)} s_A \otimes s_B.
\end{equation}
\end{claim}
\begin{proof}[Proof of the claim]
For $m \in \mathbb{N}$, let $\{e_{ij} \}_{i,j=1}^m$ be the matrix units of $M_m$, with $e_{11}, \dots , e_{mm}$ generating the diagonal matrices $D_m$.

By compactness of $Y$ and minimality of $\beta$, the space $Y$ can be written as a finite union of translates of the open support of $s_B$, which is nonempty as $s_B \neq 0$.
This means that, writing $v_h$ for the unitary in $B$ corresponding to $h \in H$, there exist $h_1,\dots,h_m \in H$, such that
\begin{equation} 
\sum_{i=1}^m v_{h_i} s_B v_{h_i}^* \sim_{(D_B \subseteq B)} 1_B.
\end{equation} 
Since $v_h s_B v^*_h \sim_{(D_B \subseteq B)} s_B$ for all $h \in H$, we can conclude that
\begin{equation} \label{eq:e11_to_1m}
1_B\otimes e_{11} \preceq_{(D_B\otimes D_m\subseteq B\otimes M_m)} s_B\otimes 1_{M_m}.
\end{equation}

Consider next the open support $\suppo(s_A) \subseteq X$. The space $X$ is infinite (since $G$ is infinite and the action is free), and therefore so is $\suppo(s_A)$ since, by compactness of $X$ and minimality of $\alpha$, the space $X$ can be written as a union of a finite number of translates of $\suppo(s_A)$.
We can therefore find nonempty pairwise disjoint open sets $U_1, \dots, U_m \subseteq \suppo(s_A)$. For any point $x_0 \in X$, by minimality of $\alpha$ there are $g_1, \dots, g_m \in G$ such that $\alpha_{g_j}(x_0) \in U_j$ for each $j=1, \dots, m$, and by continuity we can thus find an open neighborhood $V$ of $x_0$ such that $\alpha_{g_j}(V) \subseteq U_j$, for all $j= 1, \dots , m$. By Urysohn's lemma, there exists $t_A \in C(X)_+$ such that $\suppo(t_A) \subseteq V$ and
\begin{equation}
t_A(x) \le
\min_{1 \le j \le m} s_A(\alpha_{g_j}(x)), \quad x \in V.
\end{equation}

Writing $u_g$ for the unitary in $A$ corresponding to $g \in G$, by construction we have that $\sum_{j=1}^m  u_{g_j} t_A u_{g_j}^* \le s_A$. Note moreover that
\begin{equation}
\sum_{j=1}^m u_{g_j} t_A u_{g_j}^* \otimes e_{jj} \preceq_{(D_A\otimes D_m\subseteq A\otimes M_m)}\left( \sum_{j=i}^m  u_{g_j} t_A u_{g_j}^* \right) \otimes e_{11},
\end{equation}
where in particular the sequence witnessing the Cuntz subequivalence can be chosen in $\mathcal{R}\mathcal{N}_{A \otimes M_m}(D_A \otimes D_m)$ since
$u_{g_1} t_A u_{g_1}^*, \dots ,u_{g_m} t_A u_{g_m}^*$ are pairwise orthogonal positive elements of $C(X)$ (see \cite[Lemma 2.4]{LiaoTikuisis}).
Combining all this with the fact that $u_g t_A u^*_g \sim_{(D_A \subseteq A)} t_A$ for all $g \in G$, we get
\begin{equation} \label{eq:1m_to_e11}
t_A \otimes 1_{M_m} \preceq_{(D_A\otimes D_m\subseteq A\otimes M_m)} s_A \otimes e_{11}.
\end{equation}
We thus obtain
\begin{equation}\begin{array}{rcl}
t_A \otimes 1_B \otimes e_{11} &\stackrel{\mathclap{\eqref{eq:e11_to_1m}}}{\preceq}_{(D_A \otimes D_B\otimes D_m\subseteq A \otimes B\otimes M_m)}& t_A \otimes  s_B\otimes 1_{M_m} \\
& \stackrel{\mathclap{\eqref{eq:1m_to_e11}}}{\preceq}_{(D_A \otimes D_B\otimes D_m\subseteq A \otimes B\otimes M_m)}& s_A \otimes s_B \otimes e_{11},
\end{array}\end{equation}
which in turn gives
\begin{equation}
t_A \otimes 1_B  \preceq_{(D_A \otimes D_B \subseteq A \otimes B)} s_A \otimes s_B.\qedhere
\end{equation}

\end{proof}
\renewcommand\qedsymbol{$\square$}
Take $N > \max_{1\le i\le k}\|b_i\| \ge 0$. By tracial $\mathcal{Z}$-stability of $(D_A\subseteq A)$, there exists a c.p.c.\ order zero map $\psi\colon M_n\to A$ such that
\begin{enumerate}
\item $\psi(\mathcal{N}_{M_n}(D_n)) \subseteq \mathcal{N}_A(D_A)$,
\item \label{item:ZstabA2} $1_A-\psi(1_{M_n})\preceq_{(D_A\subseteq A)} t_A$,
\item  $\| [\psi(x),a_i] \|<\varepsilon /N $, for all contractions $x\in M_n$ and all $i=1,\dots,k$.
\end{enumerate}
Let us check that the c.p.c.\ order zero map $\varphi\colon M_n\to A\otimes B$ defined as $\varphi(x)\coloneqq \psi(x)\otimes 1_B$ satisfies the required conditions to verify tracial $\mathcal{Z}$-stability for the parameters specified at the beginning of the proof. Clearly we have $\varphi(\mathcal{N}_{M_n}(D_n)) \subseteq \mathcal{N}_{A \otimes B}(D_A \otimes D_B)$ and, for $i=1,\dots,k$,
\begin{equation}
\|[\varphi(x),a_i\otimes b_i]\|=\|[\psi(x),a_i]\otimes b_i\|=\|[\psi(x),a_i]\|\|b_i\|<\epsilon.
\end{equation}
Finally, since $1_{A \otimes B} - \phi(1_{M_n}) = (1_A - \psi(1_{M_n})) \otimes 1_B$, using our earlier claim at the last step below we have
\begin{equation}\begin{array}{rcl}
 1_{A \otimes B} - \phi(1_{M_n})
&\stackrel{\mathclap{\ref{item:ZstabA2}}}{\preceq}_{(D_A \otimes D_B \subseteq A \otimes B)}& t_A \otimes 1_B \\
&\preceq_{(D_A \otimes D_B \subseteq A \otimes B)} & s_A \otimes s_B,
\end{array}\end{equation}
as required.
\end{proof}

\bibliography{sbpgamma_updated}
\bibliographystyle{plain}

\end{document}